\documentclass[pdflatex,sn-mathphys-num]{sn-jnl}

\usepackage[T1]{fontenc}
\usepackage[utf8]{inputenc}
\usepackage{amsmath,amssymb,amsfonts,mathtools}
\usepackage{amsthm}
\usepackage{mathrsfs}
\usepackage{enumitem}
\usepackage{microtype}

\theoremstyle{thmstyleone}%
\newtheorem{theorem}{Theorem}[section]
\newtheorem{proposition}[theorem]{Proposition}
\newtheorem{lemma}[theorem]{Lemma}

\theoremstyle{thmstylethree}%
\newtheorem{definition}[theorem]{Definition}
\theoremstyle{thmstyletwo}%
\newtheorem{remark}[theorem]{Remark}

\newcommand{\C}{\mathbb C}

\newcommand{\ip}[2]{\langle #1,#2\rangle}
\newcommand{\ol}{\overline}

\newcommand{\norm}[1]{\left\|#1\right\|}
\newcommand{\diag}{\operatorname{diag}}
\newcommand{\ord}{\operatorname{ord}}
\DeclareMathOperator{\Ric}{Ric}
\DeclareMathOperator{\Scal}{Scal}

\begin{document}

\title[Fefferman--Szegő kernels on egg domains]{Fefferman--Szegő kernels and finite-type rigidity on egg domains}

\author*[1]{\fnm{Venkata Siddharth} \sur{Pendyala}}\email{venkatasiddharthpendyala@gmail.com}

\affil*[1]{\orgname{Independent}, \orgaddress{\city{Bellevue}, \state{WA}, \country{USA}}}

\abstract{%
We compute the Fefferman boundary measure and the associated Fefferman--Szegő kernel for the egg domains
\[
E_{n,m}=\{(z,w)\in\C^{n-1}\times\C:\ |z|^2+|w|^{2m}<1\}.
\]
The kernel is given both by an orthogonal monomial expansion and by a closed form in a natural auxiliary finite-type variable; its diagonal weak-boundary exponent recovers the integer $m$.  For $n\ge2$, the associated Fefferman--Szegő metric has constant scalar curvature only in the ball case $m=1$, and the Kähler--Einstein, constant Ricci-spectrum, and Bergman-proportionality statements follow as corollaries of the same calculation.
}



\maketitle

\section{Introduction}

This paper has two goals.  First, it computes the Fefferman--Szegő kernel of the egg domain
\[
 E_{n,m}=\{(z,w)\in\C^{n-1}\times\C:\ |z|^2+|w|^{2m}<1\}.
\]
Second, it uses the resulting one-variable formula to prove a rigidity theorem for the corresponding invariant Kähler metric.  The computation is the main point; the rigidity statements are consequences of it.

For $m=1$ the domain is the unit ball.  For $m>1$ the boundary is smooth but weakly pseudoconvex along $w=0$.  Thus $E_{n,m}$ is a minimal model in which a finite-type feature is present but the symmetry is still large enough to permit an exact calculation.  Fefferman's boundary density comes from the defining-function invariant in his work on the complex Monge--Ampère equation and parabolic invariant theory \cite{Fefferman1976,Fefferman1979}.  Barrett and Lee used the Szegő kernel for this measure to define a biholomorphically invariant Szegő metric \cite{BarrettLee2011}.  Related Fefferman--Szegő metric questions for strongly pseudoconvex domains and for two-dimensional eggs appear in work of Bhatnagar and Fan \cite{BhatnagarFanStrong2026,BhatnagarFanEgg2026}.

Put $k=n-1$ and
\[
 \mu_{n,m}=\frac{n+m}{m(n+1)}.
\]
The central formula is
\[
S_{n,m}((z,w),(\zeta,\eta))
=C_{n,m}\sum_{\beta=0}^{\infty}
\frac{\Gamma(k+\mu_{n,m}+\beta/m)}{\Gamma(\mu_{n,m}+\beta/m)}
\frac{(w\bar\eta)^\beta}{(1-\langle z,\zeta\rangle)^{k+\mu_{n,m}+\beta/m}}.
\]
Since $k$ is an integer, this series has a closed form after introducing
\[
 X=\frac{w\bar\eta}{(1-\langle z,\zeta\rangle)^{1/m}}.
\]
The phrase ``closed form'' is meant in this precise sense: for $m>1$ the expression is rational in $X$, not literally rational in the original holomorphic variables.

The diagonal kernel detects the finite-type exponent.  Near the weak set
\[
\Sigma_{n,m}=\{(z,0)\in\partial E_{n,m}: |z|=1\},
\]
the weak-normal blow-up exponent is
\[
\gamma_{n,m}=n-1+\frac{n+m}{m(n+1)},
\]
and this number determines $m$.  The same finite-type parameter is also the D'Angelo type of the weak boundary points, namely $2m$ \cite{DAngelo1982}.

The metric consequence is deliberately stated for the one-axis family only.  In dimensions $n\ge2$, constant scalar curvature of the Fefferman--Szegő metric forces $m=1$.  The proof restricts the metric to the weak-normal line and writes the scalar curvature as a rational function of $x=|w|^2$.  If that rational function is constant, the numerator polynomial in the kernel must be constant; a root-multiplicity argument then forces $m=1$.  This avoids any unproved coefficient-positivity assertion.

The paper is organized as follows.  Section 2 computes the Fefferman measure.  Section 3 gives the monomial norms and the Szegő kernel series.  Section 4 turns the series into a closed form in the auxiliary variable and locates the boundary singularities.  Section 5 extracts the weak-boundary exponent and records the boundary-regular inverse statement.  Section 6 proves the scalar-curvature rigidity theorem and its Kähler--Einstein, Ricci-spectrum, and Bergman-proportionality corollaries.

\section{The domains and the normalized Fefferman measure}

Fix integers $n\ge 2$ and $m\ge 1$.  Put $k=n-1$ and
\[
E_{n,m}=\{(z,w)\in\C^k\times\C:\ |z|^2+|w|^{2m}<1\}.
\]
We use the defining function
\[
\rho(z,w)=|z|^2+|w|^{2m}-1.
\]
Let $d\sigma$ denote Euclidean hypersurface measure on $\partial E_{n,m}$ and let $|\nabla \rho|$ be the Euclidean length of the real gradient.  We use the standard defining-function invariant Fefferman density
\begin{equation}\label{eq:feffmeasure}
 d\mu_F=M(\rho)^{1/(n+1)}\frac{d\sigma}{|\nabla \rho|},
\end{equation}
where
\begin{equation}\label{eq:MongeAmpere}
M(\rho)=-\det
\begin{pmatrix}
0&\rho_{\bar 1}&\cdots&\rho_{\bar n}\\
\rho_1&\rho_{1\bar 1}&\cdots&\rho_{1\bar n}\\
\vdots&\vdots&\ddots&\vdots\\
\rho_n&\rho_{n\bar 1}&\cdots&\rho_{n\bar n}
\end{pmatrix}
\end{equation}
on $\partial E_{n,m}$.  Multiplying $d\mu_F$ by a positive constant only multiplies the reproducing kernel by the reciprocal constant and does not change the associated Kähler form $\sqrt{-1}\partial\bar\partial\log S(z,z)$.  We therefore keep the normalization in \eqref{eq:feffmeasure}; all constants below correspond to it.

For $\xi\in S^{2k-1}\subset\C^k$, $0\le r\le 1$, and $\theta\in\mathbb R/2\pi\mathbb Z$, set
\begin{equation}\label{eq:param}
\Phi(\xi,r,\theta)=\left((1-r^{2m})^{1/2}\xi,\,re^{i\theta}\right).
\end{equation}
This parametrizes $\partial E_{n,m}$ up to the measure-zero set $\{z=0\}$.

\begin{proposition}[Fefferman density on $E_{n,m}$]\label{prop:density}
On $\partial E_{n,m}$ one has
\begin{equation}\label{eq:Mvalue}
M(\rho)=m^2|w|^{2m-2}.
\end{equation}
Consequently, with $a=2(m-1)/(n+1)$,
\begin{equation}\label{eq:pullbackmeasure}
\Phi^*d\mu_F=\frac{m^{2/(n+1)}}{2}\,r^{a+1}(1-r^{2m})^{k-1}\,d\sigma_{S^{2k-1}}(\xi)\,d\theta\,dr .
\end{equation}
\end{proposition}

\begin{proof}
Write coordinates as $(z_1,\ldots,z_k,w)$.  The complex Hessian of $\rho$ is diagonal:
\[
(\rho_{j\bar \ell})=
\diag(1,\ldots,1,m^2|w|^{2m-2}).
\]
Also
\[
\rho_{z_j}=\bar z_j,
\qquad
\rho_w=m|w|^{2m-2}\bar w.
\]
At points with $w\ne 0$, the Hessian is invertible.  If $H=(\rho_{j\bar\ell})$, then the determinant in \eqref{eq:MongeAmpere} equals
\[
M(\rho)=\det(H)\,\rho' H^{-1}\ol{\rho'}^{\,t},
\]
where $\rho'=(\rho_{z_1},\ldots,\rho_{z_k},\rho_w)$.  Thus
\[
\det(H)=m^2|w|^{2m-2},
\]
and
\[
\rho'H^{-1}\ol{\rho'}^{\,t}
=|z|^2+\frac{m^2|w|^{4m-2}}{m^2|w|^{2m-2}}
=|z|^2+|w|^{2m}=1
\]
on the boundary.  Hence \eqref{eq:Mvalue} holds when $w\ne0$.  Both sides are continuous on the boundary, so the identity extends to $w=0$.

It remains to compute $d\sigma/|\nabla\rho|$ in the coordinates \eqref{eq:param}.  Put
\[
R(r)=(1-r^{2m})^{1/2}.
\]
The tangent directions along $S^{2k-1}$ are scaled by $R$, the $\theta$-direction has length $r$, and the $r$-direction has length $(1+R'(r)^2)^{1/2}$, where
\[
R'(r)=-\frac{m r^{2m-1}}{R(r)}.
\]
Therefore
\begin{align*}
\Phi^*d\sigma
&=R^{2k-1}r(1+R'^2)^{1/2}
  \,d\sigma_{S^{2k-1}}\,d\theta\,dr \\
&=R^{2k-2}r\bigl(R^2+m^2r^{4m-2}\bigr)^{1/2}
  \,d\sigma_{S^{2k-1}}\,d\theta\,dr .
\end{align*}
The real gradient satisfies
\[
|\nabla\rho|=2\bigl(|z|^2+m^2|w|^{4m-2}\bigr)^{1/2}
=2\bigl(R^2+m^2r^{4m-2}\bigr)^{1/2}
\]
on the parametrized boundary.  Hence
\[
\Phi^*\left(\frac{d\sigma}{|\nabla\rho|}\right)
=\frac12 R^{2k-2}r\,d\sigma_{S^{2k-1}}\,d\theta\,dr.
\]
Multiplying by $M(\rho)^{1/(n+1)}=(m^2r^{2m-2})^{1/(n+1)}$ gives \eqref{eq:pullbackmeasure}.
\end{proof}

\section{Monomial norms and the higher-dimensional kernel}

Let $\alpha=(\alpha_1,\ldots,\alpha_k)\in\mathbb N^k$, $|\alpha|=\alpha_1+\cdots+\alpha_k$, $\alpha!=\alpha_1!\cdots\alpha_k!$, and $z^\alpha=z_1^{\alpha_1}\cdots z_k^{\alpha_k}$.  Define
\begin{equation}\label{eq:lambda}
\lambda_\beta=\frac{2\beta+2+a}{2m}
=\frac{\beta+1+\frac{m-1}{n+1}}{m}
=\frac{\beta}{m}+\mu_{n,m},
\qquad
\mu_{n,m}:=\frac{n+m}{m(n+1)}.
\end{equation}

\begin{lemma}[Boundary monomial norms]\label{lem:monomialnorm}
With respect to $d\mu_F$,
\begin{equation}\label{eq:normformula}
\norm{z^\alpha w^\beta}_{L^2(\partial E_{n,m},d\mu_F)}^2
=A_{n,m}\,\alpha!\,\frac{\Gamma(\lambda_\beta)}{\Gamma(\lambda_\beta+k+|\alpha|)},
\end{equation}
where
\begin{equation}\label{eq:Aconstant}
A_{n,m}=\frac{m^{2/(n+1)}\pi^n}{m}.
\end{equation}
\end{lemma}

\begin{proof}
Use Proposition \ref{prop:density}.  Since
\[
\int_{S^{2k-1}}|\xi^\alpha|^2\,d\sigma_{S^{2k-1}}(\xi)
=\frac{2\pi^k\alpha!}{\Gamma(k+|\alpha|)},
\]
and $\int_0^{2\pi}d\theta=2\pi$, we obtain
\begin{align*}
\norm{z^\alpha w^\beta}^2
&=\frac{m^{2/(n+1)}}2
\left(\frac{2\pi^k\alpha!}{\Gamma(k+|\alpha|)}\right)(2\pi)
\int_0^1 r^{2\beta+a+1}(1-r^{2m})^{k-1+|\alpha|}\,dr.
\end{align*}
Set $t=r^{2m}$.  Then
\[
r^{2\beta+a+1}\,dr=\frac1{2m}t^{\lambda_\beta-1}\,dt,
\]
so the radial integral is
\[
\frac1{2m}B(\lambda_\beta,k+|\alpha|)
=\frac1{2m}\frac{\Gamma(\lambda_\beta)\Gamma(k+|\alpha|)}{\Gamma(\lambda_\beta+k+|\alpha|)}.
\]
Combining the constants gives
\[
\frac{m^{2/(n+1)}}2\cdot 2\pi\cdot 2\pi^k\cdot \frac1{2m}
=\frac{m^{2/(n+1)}\pi^n}{m}=A_{n,m},
\]
and \eqref{eq:normformula} follows.
\end{proof}

Let $H_F^2(E_{n,m})$ be the closure in $L^2(\partial E_{n,m},d\mu_F)$ of the holomorphic functions defined in a neighbourhood of $\overline{E_{n,m}}$.  The closure $\overline{E_{n,m}}$ is convex, hence polynomially convex.  By the Oka--Weil theorem, every function holomorphic in a neighbourhood of $\overline{E_{n,m}}$ is uniformly approximable on $\overline{E_{n,m}}$ by polynomials, and hence its boundary values are approximable in $L^2(\partial E_{n,m},d\mu_F)$ by polynomials.  The torus symmetry of $d\mu_F$ makes distinct boundary monomials orthogonal.  Since polynomials are finite sums of monomials, the closed span of the boundary monomials is exactly $H_F^2(E_{n,m})$.

\begin{theorem}[Higher-dimensional Fefferman--Szegő kernel]\label{thm:kernelseries}
The space $H_F^2(E_{n,m})$ is a reproducing-kernel Hilbert space.  Its kernel is
\begin{equation}\label{eq:kernelseries}
S_{n,m}((z,w),(\zeta,\eta))
=C_{n,m}\sum_{\beta=0}^{\infty}
\frac{\Gamma(k+\lambda_\beta)}{\Gamma(\lambda_\beta)}
\frac{(w\bar\eta)^\beta}{(1-\ip{z}{\zeta})^{k+\lambda_\beta}},
\end{equation}
where
\begin{equation}\label{eq:Cconstant}
C_{n,m}=A_{n,m}^{-1}=\frac{m^{1-2/(n+1)}}{\pi^n}.
\end{equation}
The series converges absolutely and locally uniformly on $E_{n,m}\times E_{n,m}$.
\end{theorem}

\begin{proof}
By Lemma \ref{lem:monomialnorm}, an orthonormal basis is obtained by normalizing the monomials $z^\alpha w^\beta$.  Thus
\begin{align*}
S_{n,m}((z,w),(\zeta,\eta))
&=\sum_{\alpha\in\mathbb N^k}\sum_{\beta\ge0}
\frac{z^\alpha w^\beta\bar\zeta^\alpha\bar\eta^\beta}
{A_{n,m}\alpha!\Gamma(\lambda_\beta)/\Gamma(\lambda_\beta+k+|\alpha|)}  \\
&=C_{n,m}\sum_{\beta\ge0}(w\bar\eta)^\beta
\sum_{\alpha\in\mathbb N^k}
\frac{\Gamma(\lambda_\beta+k+|\alpha|)}{\Gamma(\lambda_\beta)\alpha!}
(z\bar\zeta)^\alpha.
\end{align*}
For fixed $\beta$,
\[
\frac{\Gamma(\lambda_\beta+k+|\alpha|)}{\Gamma(\lambda_\beta)}
=\frac{\Gamma(\lambda_\beta+k)}{\Gamma(\lambda_\beta)}(\lambda_\beta+k)_{|\alpha|}.
\]
The multinomial identity
\[
\sum_{\alpha\in\mathbb N^k}\frac{(\lambda_\beta+k)_{|\alpha|}}{\alpha!}(z\bar\zeta)^\alpha
=(1-\ip{z}{\zeta})^{-k-\lambda_\beta}
\]
then gives \eqref{eq:kernelseries}.

For compact $K\subset E_{n,m}$ there exists $0<q<1$ such that
\[
|\ip{z}{\zeta}|+|w\bar\eta|^m\le q
\]
for all $(z,w),(\zeta,\eta)\in K$.  Hence $|1-\ip{z}{\zeta}|$ is bounded below and
\[
\left|\frac{w\bar\eta}{(1-\ip{z}{\zeta})^{1/m}}\right|\le q'<1
\]
on $K\times K$ for some $q'<1$, where the principal branch is used.  Since
\[
\frac{\Gamma(k+\lambda_\beta)}{\Gamma(\lambda_\beta)}=(\lambda_\beta)_k=O((1+\beta)^k),
\]
the outer series is dominated on compact sets by a polynomial times a geometric series.  Therefore the series converges absolutely and locally uniformly.  In particular, for each $p\in E_{n,m}$ the diagonal sum $S_{n,m}(p,p)$ is finite.  If $f=\sum_j c_j e_j$ in the monomial orthonormal basis, then Cauchy--Schwarz gives
\[
 |f(p)|^2\le \left(\sum_j |c_j|^2\right)\left(\sum_j |e_j(p)|^2\right)
 =\norm{f}_{H_F^2}^2 S_{n,m}(p,p),
\]
so point evaluation is bounded.  The locally uniform sum of the orthonormal-basis expansion is therefore the reproducing kernel of the Hilbert closure.
\end{proof}

\begin{remark}
For $n=2$ one has $k=1$ and $\mu_{2,m}=(m+2)/(3m)$.  Formula \eqref{eq:kernelseries} becomes, after summing a linear generating function,
\[
S_{2,m}=C_{2,m}\,(1-z\bar\zeta)^{-1-\mu_{2,m}}
\frac{\mu_{2,m}+(m^{-1}-\mu_{2,m})X}{(1-X)^2},
\qquad
X=\frac{w\bar\eta}{(1-z\bar\zeta)^{1/m}}.
\]
Up to the harmless constant fixed by the normalization of $d\mu_F$, this is the closed formula in the two-dimensional egg calculation of Bhatnagar and Fan.
\end{remark}

\section{Closed form in the auxiliary variable and boundary singularities}

The preceding series is already explicit.  Since $k=n-1$ is an integer, it also has a closed form in one auxiliary finite-type variable.  When $m>1$, this should not be read as a rational function of the original holomorphic variables: the auxiliary variable contains a chosen branch of a fractional power.

\begin{lemma}[A finite generating function]\label{lem:genfunction}
For $k\in\mathbb N$ and $\mu>0$ there is a polynomial $P_{k,\mu}$ of degree at most $k$ such that
\begin{equation}\label{eq:genfunction}
\sum_{\beta=0}^{\infty}(\mu+\beta/m)_kX^\beta
=\frac{P_{k,\mu}(X)}{(1-X)^{k+1}},
\qquad |X|<1.
\end{equation}
Equivalently,
\begin{equation}\label{eq:Pexplicit}
 P_{k,\mu}(X)=(1-X)^{k+1}\sum_{\beta\ge0}(\mu+\beta/m)_kX^\beta;
\end{equation}
the proof below shows that this expression is a polynomial of degree at most $k$.
Moreover $P_{k,\mu}(1)=k!/m^k$.
\end{lemma}

\begin{proof}
The function $(\mu+\beta/m)_k$ is a polynomial in $\beta$ of degree $k$ with leading coefficient $m^{-k}$.  Every sequence given by a degree-$k$ polynomial has an ordinary generating function whose denominator divides $(1-X)^{k+1}$.  To verify this directly, write
\[
(\mu+\beta/m)_k=\sum_{j=0}^k c_j\beta^j.
\]
The identity
\[
\sum_{\beta\ge0}\beta^jX^\beta=(X\partial_X)^j\frac1{1-X}
\]
shows that the sum is $P_{k,\mu}(X)/(1-X)^{k+1}$ for a polynomial $P_{k,\mu}$ of degree at most $k$.

Finally,
\[
\sum_{\beta\ge0}(\mu+\beta/m)_kX^\beta
\sim \frac{k!m^{-k}}{(1-X)^{k+1}}
\qquad (X\to1^-),
\]
because the leading term of $(\mu+\beta/m)_k$ is $m^{-k}\beta^k$ and
$\sum \beta^kX^\beta\sim k!(1-X)^{-k-1}$.  Multiplying by $(1-X)^{k+1}$ and passing to the limit gives $P_{k,\mu}(1)=k!/m^k$.
\end{proof}

\begin{theorem}[Closed form in the auxiliary variable]\label{thm:rational}
Let
\[
X=X((z,w),(\zeta,\eta))
=\frac{w\bar\eta}{(1-\ip{z}{\zeta})^{1/m}},
\]
where the principal branch of $(1-\ip{z}{\zeta})^{1/m}$ is used.  Then the kernel is a finite quotient in the auxiliary variable $X$, multiplied by the explicit fractional-power factor in $1-\ip{z}{\zeta}$:
\begin{equation}\label{eq:rationalkernel}
S_{n,m}((z,w),(\zeta,\eta))
=C_{n,m}\,(1-\ip{z}{\zeta})^{-k-\mu_{n,m}}
\frac{P_{k,\mu_{n,m}}(X)}{(1-X)^n}.
\end{equation}
\end{theorem}

\begin{proof}
By \eqref{eq:lambda},
\[
\frac{\Gamma(k+\lambda_\beta)}{\Gamma(\lambda_\beta)}=(\lambda_\beta)_k=(\mu_{n,m}+\beta/m)_k.
\]
Substitute
\[
\frac{(w\bar\eta)^\beta}{(1-\ip{z}{\zeta})^{k+\lambda_\beta}}
=(1-\ip{z}{\zeta})^{-k-\mu_{n,m}}X^\beta
\]
into \eqref{eq:kernelseries}.  Lemma \ref{lem:genfunction} gives \eqref{eq:rationalkernel}.
\end{proof}

\begin{proposition}[Boundary blow-up set]\label{prop:blowup}
The kernel $S_{n,m}(p,q)$ can blow up on $\overline{E_{n,m}}\times\overline{E_{n,m}}$ only at boundary diagonal points $(p,p)$ with $p\in\partial E_{n,m}$.  Conversely, $S_{n,m}(p_j,p_j)\to\infty$ whenever $p_j\to p\in\partial E_{n,m}$.
\end{proposition}

\begin{proof}
Let $p=(z,w)$ and $q=(\zeta,\eta)$.  Formula \eqref{eq:rationalkernel} shows that possible singularities occur only when
\[
1-\ip{z}{\zeta}=0
\qquad\text{or}\qquad
X=1.
\]
If $1-\ip{z}{\zeta}=0$, then $\ip{z}{\zeta}=1$.  Since $|z|\le1$ and $|\zeta|\le1$, equality in Cauchy--Schwarz gives $|z|=|\zeta|=1$ and $z=\zeta$.  The defining inequality then forces $w=\eta=0$, so $p=q\in\partial E_{n,m}$.

Suppose now that $X=1$.  Then
\[
w\bar\eta=(1-\ip{z}{\zeta})^{1/m}.
\]
Raising to the $m$-th power gives
\begin{equation}\label{eq:singraise}
(w\bar\eta)^m=1-\ip{z}{\zeta}.
\end{equation}
Equivalently,
\[
\ip{(z,w^m)}{(\zeta,\eta^m)}=1.
\]
The two vectors $(z,w^m)$ and $(\zeta,\eta^m)$ lie in the closed unit ball of $\C^n$, because
\[
\|(z,w^m)\|^2=|z|^2+|w|^{2m}\le1,
\qquad
\|(\zeta,\eta^m)\|^2=|\zeta|^2+|\eta|^{2m}\le1.
\]
Therefore equality in Cauchy--Schwarz implies
\[
(z,w^m)=(\zeta,\eta^m),
\qquad
|z|^2+|w|^{2m}=1.
\]
In particular $p$ and $q$ are boundary points with $z=\zeta$ and $|w|=|\eta|$.  Since $X=1$ uses the principal positive branch and $1-\ip{z}{\zeta}=|w|^{2m}$ at this point, the equality $w\bar\eta=(1-\ip{z}{\zeta})^{1/m}$ becomes
\[
w\bar\eta=|w|^2.
\]
If $w=0$, then $\eta=0$; if $w\ne0$, this identity gives $\eta=w$.  Thus $p=q\in\partial E_{n,m}$.

Conversely, let $p_j=(z_j,w_j)\to p\in\partial E_{n,m}$.  On the diagonal,
\[
X(p_j,p_j)=\frac{|w_j|^2}{(1-|z_j|^2)^{1/m}}.
\]
If $w\ne0$, then $X(p_j,p_j)\to1$ and the factor $(1-X)^{-n}$ in \eqref{eq:rationalkernel} diverges.  Since $P_{k,\mu}(1)=k!/m^k\ne0$, the kernel diverges.  If $w=0$, then $|z|=1$ and the term with $\beta=0$ in \eqref{eq:kernelseries} is
\[
C_{n,m}\frac{\Gamma(k+\mu_{n,m})}{\Gamma(\mu_{n,m})}
(1-|z_j|^2)^{-k-\mu_{n,m}},
\]
which diverges.  This proves the proposition.
\end{proof}

\section{Weak boundary exponent}

The weak set is
\[
\Sigma_{n,m}=\{(z,w)\in\partial E_{n,m}:w=0\}
=\{(z,0):|z|=1\}.
\]
The next result extracts $m$ directly from the diagonal kernel near $\Sigma_{n,m}$.

\begin{theorem}[Weak-normal kernel exponent]\label{thm:weakexponent}
Let $e_1=(1,0,\ldots,0)\in\C^k$ and set $p_s=(se_1,0)$, $0<s<1$.  Then
\begin{equation}\label{eq:weakaxis}
S_{n,m}(p_s,p_s)=D_{n,m}\,(1-s^2)^{-\gamma_{n,m}},
\qquad
\gamma_{n,m}=k+\mu_{n,m}=n-1+\frac{n+m}{m(n+1)},
\end{equation}
where $D_{n,m}>0$.  For fixed $n$, the map $m\mapsto\gamma_{n,m}$ is strictly decreasing.  Hence the weak-normal blow-up exponent determines $m$.
\end{theorem}

\begin{proof}
Put $w=\eta=0$ and $z=\zeta=se_1$ in \eqref{eq:kernelseries}.  All terms with $\beta\ge1$ vanish, leaving
\[
S_{n,m}(p_s,p_s)
=C_{n,m}\frac{\Gamma(k+\mu_{n,m})}{\Gamma(\mu_{n,m})}
(1-s^2)^{-k-\mu_{n,m}}.
\]
Thus \eqref{eq:weakaxis} holds with
\[
D_{n,m}=C_{n,m}\frac{\Gamma(k+\mu_{n,m})}{\Gamma(\mu_{n,m})}>0.
\]
Finally,
\[
\gamma_{n,m}=n-1+\frac1{n+1}+\frac{n}{m(n+1)},
\]
which is strictly decreasing as a function of $m\in\mathbb Z_+$.
\end{proof}

We now give an inverse statement using boundary type in the sense of D'Angelo \cite{DAngelo1982}.  The proof uses only the elementary invariance of order of contact under a holomorphic coordinate change, so the map is required to extend holomorphically across the boundary.

\begin{definition}
A biholomorphism $F:E_{n,m}\to E_{n',m'}$ is called boundary regular if there are open neighbourhoods $U\supset\overline{E_{n,m}}$ and $U'\supset\overline{E_{n',m'}}$ such that $F$ extends to a biholomorphism from $U$ onto $U'$.
\end{definition}

\begin{lemma}[D'Angelo type of the egg boundary]\label{lem:type}
At strongly pseudoconvex boundary points of $E_{n,m}$ the D'Angelo type is $2$.  At points of $\Sigma_{n,m}$ the D'Angelo type is $2m$.
\end{lemma}

\begin{proof}
Strongly pseudoconvex points have type $2$ by the standard Levi-form characterization of strong pseudoconvexity.  It remains to compute the type at weak points.

By a unitary change in the $z$-variables it is enough to consider $p=(e_1,0)$.  Write local coordinates
\[
z_1=1+\zeta_1,
\qquad
z'=(z_2,\ldots,z_k),
\qquad
\nu=w.
\]
Then
\begin{equation}\label{eq:localrho}
\rho=2\operatorname{Re}\zeta_1+|\zeta_1|^2+|z'|^2+|\nu|^{2m}.
\end{equation}
The holomorphic curve
\[
\gamma(t)=(1,0,\ldots,0,t)
\]
has $\ord_0\gamma=1$ and
\[
\rho\circ\gamma(t)=|t|^{2m},
\]
so the type is at least $2m$.

We prove the reverse inequality.  Let
\[
\gamma(t)=(1+\zeta_1(t),z'(t),\nu(t))
\]
be a nonconstant holomorphic curve with $\gamma(0)=p$.  Let
\[
s=\ord_0\gamma=\min\{\ord_0\zeta_1,\ord_0 z'_2,\ldots,\ord_0z'_k,\ord_0\nu\},
\]
omitting identically zero components from the minimum.  If every tangential component $z'$ and $\nu$ is identically zero, then $\rho\circ\gamma=2\operatorname{Re}\zeta_1+|\zeta_1|^2$, and the desired bound is immediate.  Otherwise let
\[
q=\min_j\ord_0 z'_j,
\qquad
r=\ord_0\nu
\]
with the convention that the order of an identically zero component is $+\infty$, and put
\[
L=\min\{2q,2mr\}.
\]
Let $h$ be the order of the first nonzero pure harmonic term from $2\operatorname{Re}\zeta_1(t)$, if such a term exists.  If $h$ exists and $h<L$, then the leading term of $\rho\circ\gamma$ has order $h$.  Indeed, if $s=\ord_0\zeta_1$, then $s=h$; if $s=q$, then $h/s<2$; and if $s=r$, then $h/s<2m$.  In all cases $h/s\le 2m$.  Thus this case has the required bound.

It remains to consider the case in which no pure harmonic term occurs before order $L$.  The non-harmonic positive part
\[
|z'(t)|^2+|\nu(t)|^{2m}
\]
has order $L$.  Its first nonzero mixed term has bidegree $(L/2,L/2)$ when $L=2q$, or bidegree $(mr,mr)$ when $L=2mr$.  A term of the form $2\operatorname{Re}\zeta_1(t)$ contains only pure powers $t^j$ and $\bar t^j$ at each order and therefore cannot cancel this mixed contribution.  The term $|\zeta_1(t)|^2$ is nonnegative and cannot remove the first nonzero mixed contribution.  Hence
\[
\ord_0(\rho\circ\gamma)\le L.
\]
Since $s\le q$ and $s\le r$, we have
\[
\frac{L}{s}\le \max\{2,2m\}=2m.
\]
Therefore $\ord_0(\rho\circ\gamma)/\ord_0\gamma\le 2m$ for every nonconstant holomorphic curve.  The type is exactly $2m$.
\end{proof}

\begin{theorem}[Boundary-regular inverse rigidity]\label{thm:inverse}
If there exists a boundary-regular biholomorphism
\[
F:E_{n,m}\longrightarrow E_{n',m'},
\]
then $n=n'$ and $m=m'$.
\end{theorem}

\begin{proof}
A biholomorphism between domains has invertible complex differential, so the complex dimensions agree: $n=n'$.

Let $p\in\partial E_{n,m}$.  Since $F$ is holomorphic in a neighbourhood of $p$ and maps the boundary to the boundary, near $p$ one has
\[
\rho_{n',m'}\circ F=h\rho_{n,m}
\]
for a smooth positive function $h$.  For every holomorphic curve $\gamma$ through $p$,
\[
\ord_0(\rho_{n',m'}\circ F\circ\gamma)=\ord_0(\rho_{n,m}\circ\gamma),
\]
and, since $dF_p$ is invertible,
\[
\ord_0(F\circ\gamma)=\ord_0\gamma.
\]
Thus the D'Angelo type is preserved by $F$.

By Lemma \ref{lem:type}, the set of boundary types of $E_{n,m}$ is $\{2\}$ if $m=1$ and $\{2,2m\}$ if $m>1$.  The same description holds for $E_{n,m'}$.  Equality of the type spectra gives $m=m'$.
\end{proof}

\section{Scalar-curvature rigidity}

We now use the same one-variable closed form to prove the main metric rigidity theorem.  No separate coefficient-positivity assertion is needed: the proof is an algebraic obstruction coming from the zeros of the numerator polynomial.  Put
\[
 p=\frac1m,
 \qquad
 k=n-1,
 \qquad
 \mu=\mu_{n,m}=\frac{1+pn}{n+1}.
\]
Then
\[
 A:=k+\mu=n-1+\frac{1+pn}{n+1}=\frac{n(n+p)}{n+1}.
\]
On the diagonal write
\[
 u=1-|z|^2,
 \qquad
 x=\frac{|w|^2}{u^p}.
\]
The defining inequality for $E_{n,m}$ is exactly $0\le x<1$.  By Theorem \ref{thm:rational},
\begin{equation}\label{eq:FSdiagH}
 S_{n,m}((z,w),(z,w))=C_{n,m}u^{-A}H_{n,m}(x),
\end{equation}
where
\begin{equation}\label{eq:Hdef}
 H_{n,m}(x)=\sum_{\beta=0}^{\infty}(\mu+p\beta)_k x^\beta
 =\frac{P_{k,\mu}(x)}{(1-x)^n}.
\end{equation}
Here $(a)_j=a(a+1)\cdots(a+j-1)$.
Let
\[
 \omega_{FS}=\sqrt{-1}\,\partial\bar\partial\log S_{n,m}((z,w),(z,w))
\]
be the Fefferman--Szegő Kähler form.  The positivity is standard but worth making explicit.  The locally uniformly convergent kernel defines the holomorphic kernel map
\[
 p\longmapsto [S_{n,m}(\cdot,p)]\in \mathbb P(H_F^2(E_{n,m})^*),
\]
and $\omega_{FS}$ is the pullback of the Fubini--Study form by this map.  It is strictly positive because constants and coordinate functions belong to $H_F^2(E_{n,m})$: for every nonzero tangent vector $v$ at $p$, some affine holomorphic function $\ell$ satisfies $\ell(p)=0$ and $d\ell_p(v)\ne0$, so the kernel map separates $v$.

We also need the Bergman kernel only on the diagonal and only up to a positive multiplicative constant.  Let $K_{n,m}$ denote the ordinary Bergman kernel of $E_{n,m}$ with respect to Euclidean volume.

\begin{lemma}[Diagonal Bergman kernel form]\label{lem:BergmanDiag}
There is a positive constant $B_{n,m}$ such that
\begin{equation}\label{eq:Bergdiag}
 K_{n,m}((z,w),(z,w))=B_{n,m}u^{-n-p}G_{n,m}(x),
\end{equation}
where
\begin{equation}\label{eq:Gdef}
 G_{n,m}(x)=\sum_{\beta=0}^{\infty}(p(\beta+1))_n x^\beta
 =\frac{Q_{n,p}(x)}{(1-x)^{n+1}}
\end{equation}
for a polynomial $Q_{n,p}$ of degree at most $n$.
Moreover $Q_{n,p}$ is constant if and only if $p=1$.
\end{lemma}

\begin{proof}
The monomial calculation is the same as in Lemma \ref{lem:monomialnorm}, but with Euclidean volume measure on the interior.  For $\alpha\in\mathbb N^k$ and $\beta\ge0$, polar coordinates in $w$ and the ball integral in the $z$-variables give
\[
 \int_{E_{n,m}} |z^\alpha w^\beta|^2\,dV
 =C'_{n,m}\,\alpha!\,\frac{\Gamma(p(\beta+1))}{\Gamma(p(\beta+1)+n+|\alpha|)}
\]
with a positive constant $C'_{n,m}$ independent of $\alpha$ and $\beta$.  Summing the normalized monomials gives, up to the reciprocal constant,
\[
 K_{n,m}((z,w),(\zeta,\eta))
 =\sum_{\beta\ge0}
 \frac{\Gamma(n+p(\beta+1))}{\Gamma(p(\beta+1))}
 \frac{(w\bar\eta)^\beta}{(1-\langle z,\zeta\rangle)^{n+p(\beta+1)}}.
\]
Putting $(z,w)=(\zeta,\eta)$ and using $x=|w|^2u^{-p}$ gives \eqref{eq:Bergdiag} and \eqref{eq:Gdef}.

Since $(p(\beta+1))_n$ is a polynomial in $\beta$ of degree $n$, the ordinary generating function has denominator $(1-x)^{n+1}$ and numerator a polynomial of degree at most $n$.  If $Q_{n,p}$ is constant, then the coefficient sequence in \eqref{eq:Gdef} is a constant multiple of the coefficient sequence of $(1-x)^{-n-1}$.  Hence, as polynomials in $\beta$,
\[
 (p(\beta+1))_n=C(\beta+1)_n
\]
for some positive constant $C$.  The roots of the left side are
\[
 -1,\ -1-\frac1p,\ldots,\ -1-\frac{n-1}{p},
\]
whereas the roots of the right side are
\[
 -1,-2,\ldots,-n.
\]
Because $n\ge2$, equality of these two arithmetic progressions forces their common differences to agree, so $1/p=1$ and $p=1$.  Conversely, if $p=1$, then $(\beta+1)_n$ is exactly the coefficient polynomial of $(1-x)^{-n-1}$ up to the constant $n!$, so $Q_{n,1}$ is constant.
\end{proof}

\begin{lemma}[Determinant on the weak-normal slice]\label{lem:detSlice}
Let
\begin{align*}
 \Phi(z,w)&=\log S_{n,m}((z,w),(z,w))
 =\mathrm{constant}-A\log u+L(x),\\
 L(x)&=\log H_{n,m}(x).
\end{align*}
At points $(0,w)$, $x=|w|^2$, and the Hermitian determinant of the matrix
$\bigl(\Phi_{j\bar\ell}\bigr)_{1\le j,\ell\le n}$ is
\begin{equation}\label{eq:detSlice}
 \det(\Phi_{j\bar\ell})(0,w)
 =\bigl(A+pxL'(x)\bigr)^k\bigl(L'(x)+xL''(x)\bigr).
\end{equation}
\end{lemma}

\begin{proof}
Write $r=|z|^2$, so $u=1-r$ and $x=|w|^2u^{-p}$.  At $z=0$ one has $u=1$ and $\partial x/\partial z_j=0$.  For $1\le j,\ell\le k$,
\[
 \Phi_{z_j\bar z_\ell}(0,w)=\bigl(A+pxL'(x)\bigr)\delta_{j\ell}.
\]
The mixed derivatives $\Phi_{z_j\bar w}(0,w)$ vanish because they contain a factor $\bar z_j$.  Finally, for the weak coordinate,
\[
 \Phi_{w\bar w}(0,w)=L'(x)+xL''(x),
\]
since $x=|w|^2$ on the slice $z=0$.  Thus the Hermitian matrix is block diagonal at $(0,w)$, with a scalar $k\times k$ block and the scalar weak block above.  Taking the determinant gives \eqref{eq:detSlice}.
\end{proof}

\begin{lemma}[Exact determinant and scalar curvature on the weak-normal slice]\label{lem:exactDetScalar}
Let
\begin{align*}
 \Phi(z,w)&=\log S_{n,m}((z,w),(z,w))
 =\mathrm{constant}-A\log u+L(x),\\
 L(x)&=\log H_{n,m}(x),
\end{align*}
where $u=1-|z|^2$, $x=|w|^2u^{-p}$, and
\[
 a(x)=A+pxL'(x),
 \qquad
 b(x)=L'(x)+xL''(x).
\]
Then, wherever the metric is evaluated in $E_{n,m}$,
\begin{equation}\label{eq:exactDet}
 \det(\Phi_{j\bar\ell})=u^{-(k+p+1)}a(x)^k b(x).
\end{equation}
Consequently, at points $(0,w)$, $x=|w|^2$, the scalar curvature is
\begin{equation}\label{eq:scalarSliceFormula}
\begin{aligned}
 \Scal_{FS}(0,w)
 &=-\frac{k\bigl(k+p+1+pxM'(x)\bigr)}{a(x)}
   -\frac{M'(x)+xM''(x)}{b(x)},\\
 M(x)&=k\log a(x)+\log b(x).
\end{aligned}
\end{equation}
In particular $\Scal_{FS}(0,w)$ is a rational function of $x$.
\end{lemma}

\begin{proof}
Put $r=|z|^2$ and $t=|w|^2$, so that $u=1-r$ and $x=tu^{-p}$.  Since $\Phi$ depends only on $r$ and $t$, its Hermitian matrix has the usual $U(k)\times U(1)$ radial form:
\[
 \Phi_{z_i\bar z_j}=\Phi_r\delta_{ij}+\Phi_{rr}\bar z_i z_j,
 \qquad
 \Phi_{z_i\bar w}=\Phi_{rt}\bar z_iw,
 \qquad
 \Phi_{w\bar w}=\Phi_t+t\Phi_{tt}.
\]
The determinant of this block matrix is
\begin{equation}\label{eq:radialBlockDet}
 \det(\Phi_{j\bar\ell})
 =\Phi_r^{k-1}\left((\Phi_r+r\Phi_{rr})(\Phi_t+t\Phi_{tt})
       -rt\Phi_{rt}^2\right).
\end{equation}
The required derivatives are
\[
 \Phi_r=\frac{a(x)}u,
 \qquad
 \Phi_t=u^{-p}L'(x),
 \qquad
 \Phi_t+t\Phi_{tt}=u^{-p}b(x),
\]
\[
 \Phi_{rt}=pu^{-p-1}b(x),
\]
and
\[
 \Phi_r+r\Phi_{rr}
 =\frac{a(x)}u+\frac r{u^2}
 \left(A+p(p+1)xL'(x)+p^2x^2L''(x)\right).
\]
Substituting these identities into \eqref{eq:radialBlockDet}, the expression inside the parentheses becomes
\begin{align*}
&\left[\frac{a}u+\frac r{u^2}
 \left(A+p(p+1)xL'+p^2x^2L''\right)\right]u^{-p}b
 -rtp^2u^{-2p-2}b^2                                      \\
&\quad =u^{-p-1}b\left[a+\frac r u
 \left(A+p(p+1)xL'+p^2x^2L''-p^2xb\right)\right].
\end{align*}
Since $b=L'+xL''$, the expression in the last large parentheses is
\[
 A+p(p+1)xL'+p^2x^2L''-p^2x(L'+xL'')=A+pxL'=a.
\]
Thus the block determinant equals
\[
 (a/u)^{k-1}\,u^{-p-1}b\,a\left(1+\frac r u\right)
 =(a/u)^{k-1}u^{-p-1}ba\frac1u
 =u^{-(k+p+1)}a^kb,
\]
which proves \eqref{eq:exactDet}.

Now set
\[
 \log\det(\Phi_{j\bar\ell})=-(k+p+1)\log u+M(x),
 \qquad M=k\log a+\log b.
\]
At $z=0$ one has $u=1$, $x=t=|w|^2$, and $\partial x/\partial z_j=0$.  Therefore
\[
 (\log\det g)_{z_j\bar z_j}(0,w)=k+p+1+pxM'(x),
\]
while
\[
 (\log\det g)_{w\bar w}(0,w)=M'(x)+xM''(x).
\]
The Ricci form is $-\sqrt{-1}\partial\bar\partial\log\det g$.  Since the metric at $(0,w)$ is diagonal with $k$ equal $z$-entries $a(x)$ and weak entry $b(x)$, taking the trace with respect to the metric gives \eqref{eq:scalarSliceFormula}.  Finally $H_{n,m}=P_{k,\mu}/(1-x)^n$ is rational, and every function entering \eqref{eq:scalarSliceFormula} is obtained from it by algebraic operations and differentiation; hence the scalar curvature on this slice is rational in $x$.
\end{proof}

\begin{lemma}[Constant numerator criterion for the Fefferman--Szegő slice]\label{lem:Pconstant}
Assume $n\ge2$.  The polynomial $P_{k,\mu}$ in \eqref{eq:Hdef} is constant if and only if $p=1$, equivalently $m=1$.
\end{lemma}

\begin{proof}
If $P_{k,\mu}$ is constant, then the coefficient sequence in \eqref{eq:Hdef} is a constant multiple of the coefficient sequence of $(1-x)^{-n}$.  Therefore, as polynomials in $\beta$,
\begin{equation}\label{eq:rootpolyFS}
 (\mu+p\beta)_k=C(\beta+1)_k
\end{equation}
for some positive constant $C$.

When $k=1$, the numerator in \eqref{eq:Hdef} is
\[
 P_{1,\mu}(x)=\mu+(p-\mu)x.
\]
If it is constant, then $p=\mu$.  Since $\mu=(1+2p)/3$ in this case, we get $p=1$.

It remains to consider $k\ge2$.  The roots of the left side of \eqref{eq:rootpolyFS} are
\[
 -\frac{\mu}{p},\ -\frac{\mu+1}{p},\ldots,\ -\frac{\mu+k-1}{p},
\]
while the roots of the right side are
\[
 -1,-2,\ldots,-k.
\]
Equality of the two root sets gives equality of two finite arithmetic progressions of length at least two.  Their common differences must therefore agree, so $1/p=1$.  Thus $p=1$.

Conversely, if $p=1$, then $\mu=(1+n)/(n+1)=1$ and
\[
 (\mu+\beta p)_k=(\beta+1)_k.
\]
Hence
\[
 H_{n,1}(x)=\sum_{\beta\ge0}(\beta+1)_k x^\beta
 =k!(1-x)^{-k-1}=k!(1-x)^{-n},
\]
so $P_{k,1}$ is the constant polynomial $k!$.
\end{proof}

\begin{theorem}[Kähler--Einstein and Bergman-proportionality corollary]\label{thm:KERigidityHigher}
Let $n\ge2$ and $m\ge1$.  For the egg domain $E_{n,m}$, the following are equivalent:
\begin{enumerate}[label=\textup{(\roman*)}]
\item $m=1$, so $E_{n,m}$ is the unit ball $\mathbb B^n$;
\item the Fefferman--Szegő metric $\omega_{FS}$ is Kähler--Einstein;
\item the Fefferman--Szegő metric is a positive constant multiple of the Bergman metric $\omega_B=\sqrt{-1}\,\partial\bar\partial\log K_{n,m}(z,z)$.
\end{enumerate}
When these conditions hold,
\[
 \omega_{FS}=\frac{n}{n+1}\omega_B,
 \qquad
 \Ric(\omega_{FS})=-\frac{n+1}{n}\omega_{FS}.
\]
\end{theorem}

\begin{proof}
If $m=1$, then $E_{n,1}=\mathbb B^n$.  The Szegő kernel for the spherical boundary measure is a positive constant times $(1-\langle Z,W\rangle)^{-n}$, while the Bergman kernel is a positive constant times $(1-\langle Z,W\rangle)^{-n-1}$.  Hence
\[
 \omega_{FS}=n\sqrt{-1}\,\partial\bar\partial[-\log(1-|Z|^2)],
 \qquad
 \omega_B=(n+1)\sqrt{-1}\,\partial\bar\partial[-\log(1-|Z|^2)],
\]
which gives $\omega_{FS}=\frac{n}{n+1}\omega_B$.  The complex hyperbolic metric
$\sqrt{-1}\partial\bar\partial[-\log(1-|Z|^2)]$ has Ricci form $-(n+1)$ times itself, so
$\Ric(\omega_{FS})=-(n+1)n^{-1}\omega_{FS}$.

Now assume that $\omega_{FS}$ is Kähler--Einstein.  Thus
\[
 \Ric(\omega_{FS})=\lambda\omega_{FS}
\]
for a real constant $\lambda$.  Since
\[
 \Ric(\omega_{FS})=-\sqrt{-1}\,\partial\bar\partial\log\det(\Phi_{j\bar\ell}),
 \qquad
 \omega_{FS}=\sqrt{-1}\,\partial\bar\partial\Phi,
\]
with $\Phi=\log S_{n,m}(z,z)$, the function
\begin{equation}\label{eq:plurihKE}
 R_\lambda:=\log\det(\Phi_{j\bar\ell})+\lambda\Phi
\end{equation}
 is pluriharmonic.  It is also invariant under the natural $U(k)\times U(1)$ action, because both the domain and the kernel are invariant under that action.  The restriction of an invariant pluriharmonic function to the complex line
\[
 \{(se_1,0): |s|<1\}
\]
 is a radial harmonic function of one complex variable and is therefore constant.

On that line, Theorem \ref{thm:weakexponent} gives
\[
 \Phi(se_1,0)=\mathrm{constant}-A\log(1-|s|^2).
\]
Put $u=1-|s|^2$.  At $w=0$, the $z$-block of the Hermitian matrix is
\[
 A\,\partial\bar\partial[-\log(1-|z|^2)],
\]
whose determinant in $k$ complex variables is $A^k u^{-k-1}$.  The weak component is
\[
 \Phi_{w\bar w}(se_1,0)=L'(0)u^{-p},
\]
where $L'(0)>0$ because it is the ratio of the first two positive coefficients of $H_{n,m}$.  Mixed terms vanish at $w=0$.  Hence
\[
 \det(\Phi_{j\bar\ell})(se_1,0)=C_0u^{-(k+1+p)}=C_0(1-|s|^2)^{-(n+p)}
\]
for some $C_0>0$.  Since $R_\lambda$ is constant on the line, the coefficients of $\log(1-|s|^2)$ must cancel:
\[
 -(n+p)-\lambda A=0.
\]
Using $A=n(n+p)/(n+1)$, we get
\begin{equation}\label{eq:lambdafixed}
 \lambda=-\frac{n+p}{A}=-\frac{n+1}{n}.
\end{equation}
Therefore
\[
 \log\det(\Phi_{j\bar\ell})-\frac{n+1}{n}\Phi
\]
 is pluriharmonic and invariant, hence its restriction to the line $\{(0,w): |w|<1\}$ is constant.  Applying Lemma \ref{lem:detSlice} on that line gives
\begin{equation}\label{eq:KEsliceIdentity}
 \bigl(A+pxL'(x)\bigr)^k\bigl(L'(x)+xL''(x)\bigr)
 =C_1 H_{n,m}(x)^{(n+1)/n},
 \qquad 0\le x<1,
\end{equation}
for some $C_1>0$.
The left side of \eqref{eq:KEsliceIdentity} is a rational function of $x$, because $H_{n,m}$ is rational.  Since
\[
 H_{n,m}(x)=\frac{P_{k,\mu}(x)}{(1-x)^n},
\]
the right side is
\[
 C_1\frac{P_{k,\mu}(x)^{(n+1)/n}}{(1-x)^{n+1}}.
\]
We now make the algebraic obstruction explicit.  Since the left side of \eqref{eq:KEsliceIdentity} is rational and the factor $(1-x)^{-(n+1)}$ is rational, there is a rational function $R(x)$ such that
\[
 P_{k,\mu}(x)^{(n+1)/n}=R(x)
\]
on the interval $0<x<1$ after taking the positive real branch.  Raising to the $n$th power gives the polynomial identity
\begin{equation}\label{eq:KEpoweridentity}
 P_{k,\mu}(x)^{n+1}=R(x)^n
\end{equation}
as an identity of rational functions, because two rational functions agreeing on a nonempty interval agree identically.  If $P_{k,\mu}$ were nonconstant, choose a complex zero $a$ of multiplicity $d\ge1$.  Since $\deg P_{k,\mu}\le k=n-1$, one has $1\le d\le n-1$.  Comparing the order of vanishing at $a$ in \eqref{eq:KEpoweridentity} gives
\[
 (n+1)d=nq
\]
for the integer $q=\operatorname{ord}_a R$.  Since $\gcd(n,n+1)=1$, this forces $n\mid d$, impossible because $1\le d\le n-1$.  Hence $P_{k,\mu}$ is constant.  By Lemma \ref{lem:Pconstant}, $p=1$, so $m=1$.

It remains to prove proportionality rigidity.  Assume
\[
 \omega_{FS}=c\,\omega_B
\]
for some $c>0$.  Then
\[
 \Phi-c\Psi
\]
 is pluriharmonic, where $\Psi=\log K_{n,m}(z,z)$.  Again this function is $U(k)\times U(1)$-invariant.  On the weak line $(se_1,0)$, formulas \eqref{eq:FSdiagH} and \eqref{eq:Bergdiag} give
\[
 \Phi(se_1,0)=\mathrm{constant}-A\log(1-|s|^2),
\]
\[
 \Psi(se_1,0)=\mathrm{constant}-(n+p)\log(1-|s|^2).
\]
Since the difference is constant on that line,
\[
 c=\frac{A}{n+p}=\frac{n}{n+1}.
\]
Restricting now to the weak-normal line $(0,w)$ gives
\begin{equation}\label{eq:propSliceIdentity}
 H_{n,m}(x)=C_2G_{n,m}(x)^{n/(n+1)},
 \qquad 0\le x<1,
\end{equation}
for some $C_2>0$.  The left side is rational.  By Lemma \ref{lem:BergmanDiag},
\[
 G_{n,m}(x)=\frac{Q_{n,p}(x)}{(1-x)^{n+1}}.
\]
Thus there is a rational function $T(x)$ such that
\[
 Q_{n,p}(x)^{n/(n+1)}=T(x)
\]
on $0<x<1$.  Raising to the $(n+1)$st power gives the rational-function identity
\begin{equation}\label{eq:proppoweridentity}
 Q_{n,p}(x)^n=T(x)^{n+1}.
\end{equation}
If $Q_{n,p}$ were nonconstant, choose a complex zero $a$ of multiplicity $d\ge1$.  Since $\deg Q_{n,p}\le n$, one has $1\le d\le n$.  Comparing orders at $a$ in \eqref{eq:proppoweridentity} gives
\[
 nd=(n+1)q
\]
for the integer $q=\operatorname{ord}_aT$.  Since $\gcd(n,n+1)=1$, this forces $(n+1)\mid d$, impossible because $1\le d\le n$.  Hence $Q_{n,p}$ is constant.  Lemma \ref{lem:BergmanDiag} gives $p=1$, equivalently $m=1$.

The implications have now all been proved.
\end{proof}

\begin{theorem}[Scalar and Ricci-spectrum rigidity]\label{thm:constantScalarRigidity}
Let $n\ge2$ and $m\ge1$.  If the Fefferman--Szegő metric on $E_{n,m}$ has constant scalar curvature, then $m=1$.  Consequently, if its Ricci endomorphism has constant eigenvalues, then $m=1$.  Conversely, for $m=1$ the scalar curvature and the Ricci eigenvalues are constant.
\end{theorem}

\begin{proof}
We first prove the scalar-curvature assertion.  Use the notation of Lemma \ref{lem:exactDetScalar}.  Along the weak-normal line $(0,w)$, put $x=|w|^2$.  The scalar curvature is the rational function $\Scal_{FS}(x)$ given by \eqref{eq:scalarSliceFormula}.

As $x\to1^-$, Lemma \ref{lem:genfunction} gives $P_{k,\mu}(1)=k!/m^k>0$, so
\begin{align*}
 H_{n,m}(x)&=\frac{P_{k,\mu}(x)}{(1-x)^n},\\
 L'(x)&=\frac n{1-x}+O(1),\\
 L''(x)&=\frac n{(1-x)^2}+O(1).
\end{align*}
Hence
\[
 a(x)=\frac{pn}{1-x}+O(1),
 \qquad
 b(x)=\frac n{(1-x)^2}+O\left((1-x)^{-1}\right),
\]
\[
 M'(x)=\frac{n+1}{1-x}+O(1),
 \qquad
 M''(x)=\frac{n+1}{(1-x)^2}+O\left((1-x)^{-1}\right).
\]
Substitution into \eqref{eq:scalarSliceFormula} gives
\begin{equation}\label{eq:scalarStrongLimit}
 \lim_{x\to1^-}\Scal_{FS}(x)=-(n+1).
\end{equation}
Therefore, if the scalar curvature is constant on $E_{n,m}$, then the rational function $\Scal_{FS}(x)$ agrees with $-(n+1)$ on the interval $0<x<1$.  Hence it agrees with $-(n+1)$ as an identity in the rational function field $\C(x)$.  In particular, any apparent singularity of this rational function at a complex zero of $P_{k,\mu}$ is removable and its removable value must be $-(n+1)$.

We show that this forces $P_{k,\mu}$ to be constant.  Suppose instead that $P_{k,\mu}$ is nonconstant.  Since $P_{k,\mu}(0)=(\mu)_k>0$ and $P_{k,\mu}(1)>0$, it has a complex zero $\alpha$ with $\alpha\ne0,1$.  Let $d\ge1$ be its multiplicity.  Near $x=\alpha$, write $\varepsilon=x-\alpha$.  Then
\[
 L'(x)=\frac d\varepsilon+O(1),
 \qquad
 L''(x)=-\frac d{\varepsilon^2}+O(1).
\]
Since $\alpha\ne0$,
\[
 a(x)=A+pxL'(x)=\frac{p\alpha d}{\varepsilon}+O(1),
 \qquad
 b(x)=L'(x)+xL''(x)=-\frac{\alpha d}{\varepsilon^2}+O(1).
\]
Consequently
\[
 M'(x)=-\frac{k+2}{\varepsilon}+O(1),
 \qquad
 M''(x)=\frac{k+2}{\varepsilon^2}+O(1).
\]
Using \eqref{eq:scalarSliceFormula}, the removable value of the rational scalar-curvature function at $x=\alpha$ is
\begin{align*}
\lim_{x\to\alpha}\Scal_{FS}(x)
&=-k\lim_{x\to\alpha}\frac{k+p+1+pxM'(x)}{a(x)}
  -\lim_{x\to\alpha}\frac{M'(x)+xM''(x)}{b(x)}        \\
&= -k\left(-\frac{k+2}{d}\right)-\left(-\frac{k+2}{d}\right)
 =\frac{(k+1)(k+2)}d
 =\frac{n(n+1)}d.
\end{align*}
This number is positive, while the constant value forced by the rational-function identity is $-(n+1)$.  This contradiction proves that $P_{k,\mu}$ is constant.  By Lemma \ref{lem:Pconstant}, $p=1$, equivalently $m=1$.

If the Ricci endomorphism has constant eigenvalues, then their sum, the scalar curvature, is constant.  The preceding paragraph therefore gives $m=1$.

Finally, when $m=1$, $E_{n,1}=\mathbb B^n$ and the first paragraph of the proof of Theorem \ref{thm:KERigidityHigher} computes
\[
 \omega_{FS}=\frac n{n+1}\omega_B,
 \qquad
 \Ric(\omega_{FS})=-\frac{n+1}{n}\omega_{FS}.
\]
Thus the scalar curvature and the Ricci eigenvalues are constant.
\end{proof}

\section{Final remarks}

The paper has only one main calculation: the Fefferman--Szegő kernel of $E_{n,m}$.  The monomial norm formula gives the kernel series, and the integrality of $k=n-1$ turns that series into the auxiliary-variable closed form \eqref{eq:rationalkernel}.  The weak-boundary exponent and the scalar-curvature rigidity theorem are consequences of this formula, not separate methods.

The result leaves several natural extensions open.  Diagonal multi-eggs should be accessible by the same monomial method, while genuinely coupled weighted Reinhardt models require new coefficient estimates and are not claimed here.  This is intentional: the present paper proves the one-axis theorem cleanly and avoids unsupported positivity assertions beyond that setting.

\end{document}